\documentclass[12pt,a4paper]{article}
\usepackage[cp1251]{inputenc}
\usepackage[russian]{babel}
\usepackage{amsmath,amssymb,amsthm}
\usepackage{graphicx}
\usepackage{amsfonts,graphicx, amsthm,amsmath,amssymb}
\usepackage{hyperref}

\newtheorem{theorem}{Теорема}

\newtheorem{proposition}{Предложение}
\theoremstyle{definition}
\newtheorem{remark}{Замечание}
\newtheorem{definition}{Определение}

\newcommand{\eps}{\varepsilon}

\begin{document}
\date{}
\title{Эффект Джозефсона и быстро-медленные системы}
\author{
В.~А.~Клепцын\\CNRS, Institut de Recherche Math\'ematique de Rennes \\victor.kleptsyn@univ-rennes1.fr\\
О.~Л.~Ромаскевич\\механико-математический факультет МГУ\\olga.romaskevich@gmail.com\\
И.~В.~Щуров\\НИУ ВШЭ\\ilya@schurov.com
}
\maketitle
\begin{abstract}
In order to model the processes taking place in systems with Josephson contacts, a differential equation
on a torus with three parameters is used. One of the parameters of the system can be considered small
and the methods of the fast-slow systems theory can be applied. The properties of the phase-lock areas
– the subsets in the parameter space, in which the changing of a current doesn’t affect the voltage — are
important in practical applications. The phase-lock areas coincide with the Arnold tongues of a Poincare
map along the period. A description of the limit properties of Arnold tongues is given. It is shown that
the parameter space is split into certain areas, where the tongues have different geometrical structures
due to fast-slow effects. An efficient algorithm for the calculation of tongue borders is elaborated. The
statement concerning the asymptotic approximation of borders by Bessel functions is proven. The
presented results were obtained jointly with V. Kleptsyn, D. Filimonov, A. Klimenko and A. Glutsyuk

-----

Для моделирования процессов, происходящих в системах с джозефсоновскими контактами, используется дифференциальное уравнение на торе с тремя параметрами. Малость одного из них позволяет использовать методы теории быстро-медленных систем для изучения динамики этого уравнения. С точки зрения приложений, интерес представляют свойства зон захвата фазы -- областей в пространстве параметров, в которых изменение силы тока не меняет напряжение. Зоны захвата фазы совпадают с языками Арнольда для отображения Пуанкаре за период. Получено описание предельных свойств языков Арнольда. Показано, что пространство параметров разбивается на несколько областей, в которых языки имеют различную геометрию,
обусловленную быстро-медленными эффектами. Разработан эффективный алгоритм построения границ языков. Доказано утверждение об асимптотическом приближении границ функциями Бесселя при достаточно большом значении амплитуды тока смещения.
УДК 517.928.4, 517.925.42, 538.945
\end{abstract}

\textbf{Ключевые слова}: быстро-медленные системы, языки Арнольда, захват фазы, эффект Джозефсона
\section{Введение}
Мы изучаем дифференциальное уравнение с тремя вещественными параметрами, использующееся в физике сверхпроводников:

\begin{equation}\label{eq:Joseph}
\frac{dx}{d\tau}=\cos x +a + b \cos \mu \tau
\end{equation}
Сделаем замену времени $t=\mu\tau$.  Уравнение примет вид
\begin{equation}\label{eq:Joseph_tau}
\frac{dx}{dt}=\frac{\cos x +a + b \cos t}{\mu},
\end{equation}
эквивалетный системе
\begin{equation}\label{eq:Joseph_torus}
\left\{
\begin{array}{l}
x'=\cos x +a + b \cos t\\
t'=\mu\\
\end{array}
\right.
\end{equation}
где штрихом обозначается дифференцирование по $\tau$. Поскольку правая часть $2\pi$-периодична по $x$ и $t$, можно считать, что фазовым пространством системы~\eqref{eq:Joseph_torus} является двумерный тор~$\mathbb T^2=\mathbb R^2/(2\pi \mathbb Z^2)$.

Уравнения (\ref{eq:Joseph}-\ref{eq:Joseph_torus}) имеют одинаковые фазовые портреты и сводятся друг к другу заменой времени.

Уравнение $\eqref{eq:Joseph}$ и его обобщения активно используeтся для моделирования эффекта Джозефсона, который состоит в возможности появления сверхпроводящей компоненты для тока, протекающего через тонкий слой диэлектрика между двумя сверхпроводниками (такая конструкция называется джозефсоновским контактом). Впервые такой эффект был теоретически предсказан двадцатидвухлетним аспирантом Кембриджского университета Брайаном Джозефсоном в $1962$ году \cite{J62}, а через год был наблюден экспериментально американскими физиками П. Андерсоном и Дж. Роуэллом \cite{AR63}.

При рассмотрении джозефсоновского контакта при заданном внешнем контролируемом токе $i(t)$,  мы интересуемся зависимостью среднего напряжения $v$ от среднего тока $i$. График $v=v(i)$ называется \emph{вольт-амперной характеристикой}. Строго говоря, здесь $i$, $v$ суть безразмерные величины, отвечающие току и напряжению (переход к физическим величинам и более подробное физическое описание см. в \cite{T,LU,L}).

Экспериментально было обнаружено, что график вольт-амперной характеристики имеет ступеньки: на некотором начальном отрезке по $i$ выполнено $v=0$ (наблюдается сверхпроводимость), далее, после малого по $\mu$ промежутка на оси $i$ появляется новая ступень графика и так далее, сам график лежит под графиком параболы $v=\sqrt{i}$.
Ступеньки этого графика $\left\{i: v(i) = \mathrm{const}\right\}$ называются \emph{ступеньками Шапиро}. При этом ступеньки Шапиро в эксперименте проявляются лишь дискретно, при целых (с точностью до общей мультипликативной константы) значениях среднего напряжения.

Уравнение \eqref{eq:Joseph} представляет собой резистивную модель перехода Джозефсона с малой емкостью и синусоидальной токо-фазовой зависимостью (синусоидальным током смешения). В наших рассмотрениях после линейной замены координат ток смещения есть $I = a+ b \cos t$, параметр $a$ отвечает за среднее значение тока $i$, параметр $b$~--- за амплитуду. При этом переменная $x$ соответствует квантово-механической величине, характеризующей разность фаз. Производная  $\dot{x}$ есть уже макрофизическая величина~--- напряжение между пластинками сверхпроводников. Параметр $\mu$ играет роль отношения частоты внешнего сигнала к внутренней частоте контакта. Более строгое описание см. в \cite{MMI}.

Как было отмечено выше, физический интерес представляет зависимость среднего напряжения в течение больших промежутков времени от силы тока. В рассматриваемой модели среднее напряжение равно \emph{числу вращения} отображения Пуанкаре системы~\eqref{eq:Joseph_torus} с фиксированной пространственной окружности $t=0$ на себя. Нас интересует зависимость числа вращения от параметров $a$, $b$. Существование зон захвата фазы (языков Арнольда)~--- областей в пространстве параметров, при которых число вращения локально постоянно~--- является хорошо известным фактом теории отображений окружности. Оказывается, именно этот эффект соответствует физическому явлению~--- появлению ступенек Шапиро. Нас интересует геометрическая структура зон резонансного захвата.

Основные результаты, представленные в настоящей работе:
\begin{itemize}
	\item Доказано асимптотическое представление границ зон захвата фазы (языков Арнольда) при больших значениях амплитуды тока смещения.
	\item Объяснена связь между структурой языков при малых $\mu$ и геометрическими свойствами уравнения~\eqref{eq:Joseph} с точки зрения теории быстро-медленных систем.
	\item Построены области в пространстве параметров, «почти полностью» (за исключением экспоненциально узких по $\mu$ участков) занятые областями захвата при малых $\mu$.
	\item Приведен эффективный алгоритм вычисления границ языков для малых значений параметра $\mu$ (вплоть до $\mu$ порядка 0.01).
\end{itemize}

Структура работы такова. В разделе~\ref{ss:prop} приводятся известные факты об уравнениях~\eqref{eq:Joseph}, необходимые для дальнейшего. В разделе~\ref{sec:mainresults} приводятся и поясняются основные результаты работы.

\section{Число вращения и фазовый захват}\label{ss:prop}
В этом разделе мы приводим обзор известных результатов об уравнениях~\eqref{eq:Joseph}--\eqref{eq:Joseph_torus}.
\subsection{Число вращения уравнения на торе}
Мы напомним основные понятия теории уравнений на торе и отображений окружности, восходящей к Пуанкаре. Материал этого параграфа классический, см.~\cite{Bible,Arnold}.

Рассмотрим поле направлений на двумерном торе без особых точек. В подходящей системе координат оно записывается в виде
\begin{equation}\label{eq:anytorus}
	\frac{dx}{dt}=f(x,t),\quad (x,t)\in \mathbb T^2
\end{equation}
Заметим, что уравнение~\eqref{eq:Joseph_tau} является частным случаем уравнения~\eqref{eq:anytorus}. Для такого уравнения определено отображение Пуанкаре $P$ (отображение первого возвращения) с трансверсали $t=0$ на себя, совпадающее с отображением фазового потока за период. Пусть $\widetilde{P}$~--- поднятие $P$ на универсальную накрывающую.
\begin{definition}\label{def:rho}
	\emph{Числом вращения} системы~\eqref{eq:anytorus} (или его отображения Пуанкаре $P$) называется следующая величина:
\begin{equation}\label{eq:rho}
	\rho_P=\frac{1}{2\pi}\lim_{n \rightarrow \infty} \frac{\widetilde{P}^{n}(x_0)-x_0}{2 \pi n}=\frac{1}{2\pi}\lim_{t\to\infty}\frac{x(t)-x(0)}{t},
\end{equation}
где $x(t)$ задает некоторую фазовую траекторию уравнения~\eqref{eq:anytorus}, поднятую на универсальную накрывающую, $P^n$~--- $n$-я композиционная степень отображения $P$.
\end{definition}

Пуанкаре доказал, что число вращения всегда существует и не зависит от выбора стартовой точки $x_0$ (или траектории $x$). Оно характеризует средний наклон фазовых кривых уравнения~\eqref{eq:anytorus} или «средний угол поворота» для его отображения Пуанкаре. Для всякого отображения окружности, число вращения является важной характеристикой его динамических свойств. В частности, имеет место следующая дихотомия.

\begin{itemize}
	\item Число вращения $\rho_P$ диффеоморфизма $P$ \emph{рационально} ($\rho_P=\frac{p}{q}$) тогда и только тогда, когда у диффеоморфизма $P$ имеется периодическая орбита периода $q$
	\item При достаточной гладкости $P$ (а именно, $P \in C^{1+\varepsilon}$), число вращения $\rho_P$ иррационально тогда и только тогда, когда отображение $P$ в подходящих координатах является поворотом окружности на (иррациональный) угол $\rho_P$, то есть отображением $x\mapsto x+\rho_P \pmod 1$ (это классическая теорема Данжуа, см.~\cite{Denjoi,Bible}). В этом случае все его орбиты всюду плотны (динамика минимальна).
\end{itemize}

Пусть число вращения рационально и имеется периодическая орбита перида $q$. Это
означает, что для некоторого $x_0$
\begin{equation}\label{eq:eqx0}
	P^{q}(x_0)=x_0.
\end{equation}
Иными словами, точка $x_0$ является неподвижной точкой отображения $P^q$.
Производная $(P^q)'(x_0)$ называется \emph{мультипликатором} соответствующей
периодической точки. Пусть мультипликатор отличен от нуля (это требование
выполняется для \emph{типичного} отображения $P$ при подходящем определении
типичности). В этом случае по теореме о неявной функции, применяемой к
уравнению~\eqref{eq:eqx0}, у малого возмущения отображения $P$ также будет
периодическая точка периода $q$. Иными словами, в типичном случае периодические
орбиты сохраняются при малом шевелении. Отсюда нетрудно вывести, что число
вращения в этом случае также сохраняется.

Напротив, если число вращения достаточно гладкого отображения окружности
иррационально, и следовательно отображение сопряжено повороту, его число вращения можно изменить сколь
угодно малым возмущением. (Отображение поворота на угол $\alpha+\eps$ близко к
отображению поворота на угол $\alpha$ при малых $\eps$, но имеет число вращения
$\alpha+\eps\ne \alpha$ для $\eps\ne 0$.)

\subsection{Ступеньки Шапиро и области резонансного захвата}
Вернёмся к уравнению~\eqref{eq:Joseph_tau}. Пусть $P=P_{a, b, \mu}$~---
отображение Пуанкаре с трансверсали $t=0$ на
себя для \eqref{eq:Joseph_tau}. Заметим, что оно совпадает с отображением Пуанкаре для системы~\eqref{eq:Joseph_torus}, поскольку у них одинаковые фазовые портреты. Пусть в определении~\ref{def:rho} отображение $\widetilde{P}$~--- это поднятие~$P$
на универсальную накрывающую, непрерывно зависящее от параметров $a$, $b$,
$\mu$. В этом случае число вращения отображения $P$ также будет функцией от $a$,
$b$, $\mu$. Нетрудно видеть, что оно имеет физический смысл среднего напряжения $v$
за длительный промежуток времени. Действительно, если $\dot x$ есть напряжение в некоторый фиксированный момент времени, среднее напряжение за большой
промежуток времени для системы с джозефсоновским контактом есть $\frac{1}{T}
\int_0^T \dot{x}\,dt$, что при переходу к пределу и представляется числом вращения.

Зависимость числа вращения от параметра $a$ при фиксированном $b$ в этом случае совпадает с \emph{вольт-амперной характеристикой} джозефсоновского контакта. Как было показано выше, на графике этой зависимости могут возникать ступеньки в точках с рациональными числами вращения. Они действительно наблюдаются в экспериментах и называются~\emph{ступеньками Шапиро}.

\begin{definition}
\emph{Язык Арнольда} есть замкнутое множество в пространстве параметров с непустой внутренностью, являющееся множеством уровня $\rho_P$. Будем говорить, что происходит \emph{захват фазы} (\emph{резонансный захват}) для некоторого $\rho_P$, если существует соответствующий язык Арнольда.
\end{definition}

Впервые идея рассматривать области постоянного числа вращения в пространстве параметров для конечнопараметрических семейств отображений появилась у В.~И.~Арнольда~\cite{Arnold}. Он также заметил, что теорема Данжуа влечет отсутствие языков для иррациональных чисел вращения. В то же время, в типичном случае захват фазы происходит для всех рациональных чисел вращения.

Уравнение \eqref{eq:Joseph_tau} предоставляет парадоксальный с этой точки зрения пример семейства, в котором захват фазы наблюдается лишь для целых чисел вращения (при возмушении, однако, этот эффект пропадает~\cite{MMI}). Это свойство соответствует физическому эффекту появления ступенек Шапиро лишь при целочисленных значениях напряжения. В этом случае также число вращения равняется предельному среднему напряжению за 1 период, и последнее, соответствено, не меняется при малом изменении параметров (наблюдается «квантование числа вращения»~\cite{BKTq}, \cite{BKTq2}).

Математическое объяснение рождения языков только для целых чисел вращения приводится в следующем параграфе.

\subsection{Мёбиусовость и квантование числа вращения}
Зафиксируем некоторое $\mu >0$. Мы будем изучать структуру языков Арнольда на плоскости параметров $(a,b)$. Уравнение \eqref{eq:Joseph} сопряжено уравнению Риккати  с $2 \pi$ -периодическими коэффициентами посредством замены $u = \tan \frac{x}{2}$. Как известно, отображение Пуанкаре для уравнения Риккатти дробно-линейно \cite{E}. Соответственно, отображение Пуанкаре $P_{a,b, \mu}$ сопряжено дробно-линейному (мёбиусовому). Этот факт был замечен Футом в \cite{Foote} в связи с «велосипедной динамикой» и позже переоткрыт в \cite{BKTq,US} для уравнения, моделируюшего эффект Джозефсона. Ясно, что сопряжение не влияет на факт существования неподвижных точек и периодических орбит.

Нетождественное дробно-линейное отображение может иметь \emph{ноль}, \emph{одну} или \emph{две} неподвижных точки на вещественной оси (так как это решения квадратного уравнения $P(z)=z$). В этом случае оно называется соответственно \emph{эллиптическим}, \emph{параболическим} или \emph{гиперболическим}.

\begin{proposition}[\cite{BKTq, US}]\label{prop:only_integer}
Для уравнения \eqref{eq:Joseph_torus} захват фазы происходит лишь для $\rho_P \in \mathbb{Z}$.
\end{proposition}

\begin{proof}
	Действительно, если $\rho_{P}=\frac{p}{q} \notin \mathbb{Z}$, то отображение $P$ в подходящих мебиусовых координатах~--- эллиптическое. Рассматривая его как отображение единичнго круга, заменой координат его можно превратить в поворот.
Далее работают аналогичные случаю иррационального числа вращения рассуждения: число вращения поворота изменяется сколь угодно малым возмущением.
\end{proof}

Таким образом, языки Арнольда возникают лищь при $\rho \in \mathbb{Z}$. Точка $(a,b)$ лежит во внутренности языка Арнольда тогда и только тогда, когда соответствующее отображение $P_{a,b, \mu}$~--- гиперболическое. Точка $(a,b)$ лежит на границе языка Арнольда тогда и только тогда, когда соответствующее отображение параболическое (или тождественное). Действительно, границе языка Арнольда соответствует случай неподвижных точек, устраняющихся малым шевелением, то есть имеющих нулевую производную. Такие точки являются параболическими. При движении по кривой, начинающейся внутри языка Арнольда к его границе, действительные корни уравнения $P(z)=z$ сливаются в один корень, а при выходе из языка переходят в комплексно сопряженные точки.

\subsection{Симметрии и условия на границы}\label{sssec:sym}
Помимо мёбиусовсти ещё одним важным свойством уравнения
\eqref{eq:Joseph_torus} является обратимость динамики, то есть сохранение фазовых кривых при отображении
$(t,x) \mapsto (-t, -x)$.  Мебиусовость и центральная симметрия фазовых кривых
вместе дают аналитическое описание границ языков Арнольда в терминах
отображения Пуанкаре.

Действительно, пусть $P_{a,b, \mu} \neq \mathrm{id}$ и $(a,b)$ лежит на границе некоторого языка Арнольда при фиксированном $\mu$. Тогда $\rho_P \in \mathbb{Z}$ (см. предложение \ref{prop:only_integer}). В этом случае $P$ имеет неподвижную точку. Фазовые кривые сохраняются при центральной симметрии, и значит эта неподвижная точка отображения Пуанкаре должна переходить в неподвижную точку отображения Пуанкаре под действием симметрии $x\mapsto -x$ на окружности. Поскольку у параболического дробно-линейного отображения неподвижная точка единственна, она обязана переходить в себя. Значит, она обязана удовлетворять уравнению $x=-x$ на окружности. Существуют две точки, удовлетворяющие этому уравнению: $0$ и $\pi$. Таким образом, границы языка Арнольда с числом вращения $k \in \mathbb{Z}$ суть две аналитических кривые $a_{0,k}$ и $a_{\pi,k}$, задающиеся условиями

\begin{align*}
a=a_{0,k}(b) \Leftrightarrow P_{a,b, \mu}(0)=0 \\
a=a_{\pi,k}(b) \Leftrightarrow P_{a,b, \mu}(\pi)=\pi
\end{align*}

Численные эксперименты показывают (см. рис. \ref{intersections}), что границы одного языка Арнольда пересекаются друг с другом в счетном числе точек. Эти точки мы будем называть \emph{перемычками языков}. Также на картинке видно, что  перемычки одного языка лежат на одной вертикальной прямой $a=k \mu, k=\rho_P$. Существование счетного числа перемычек следует из результата о бесселевом приближении границ (см. параграф~\ref{subsection:Bessel} ниже, подробности в~\cite{ORK}). Целочисленность абсцисс перемычек была доказана совсем недавно в~\cite{Lesha}.

\begin{figure}
	\begin{center}
		\includegraphics[scale=0.11]{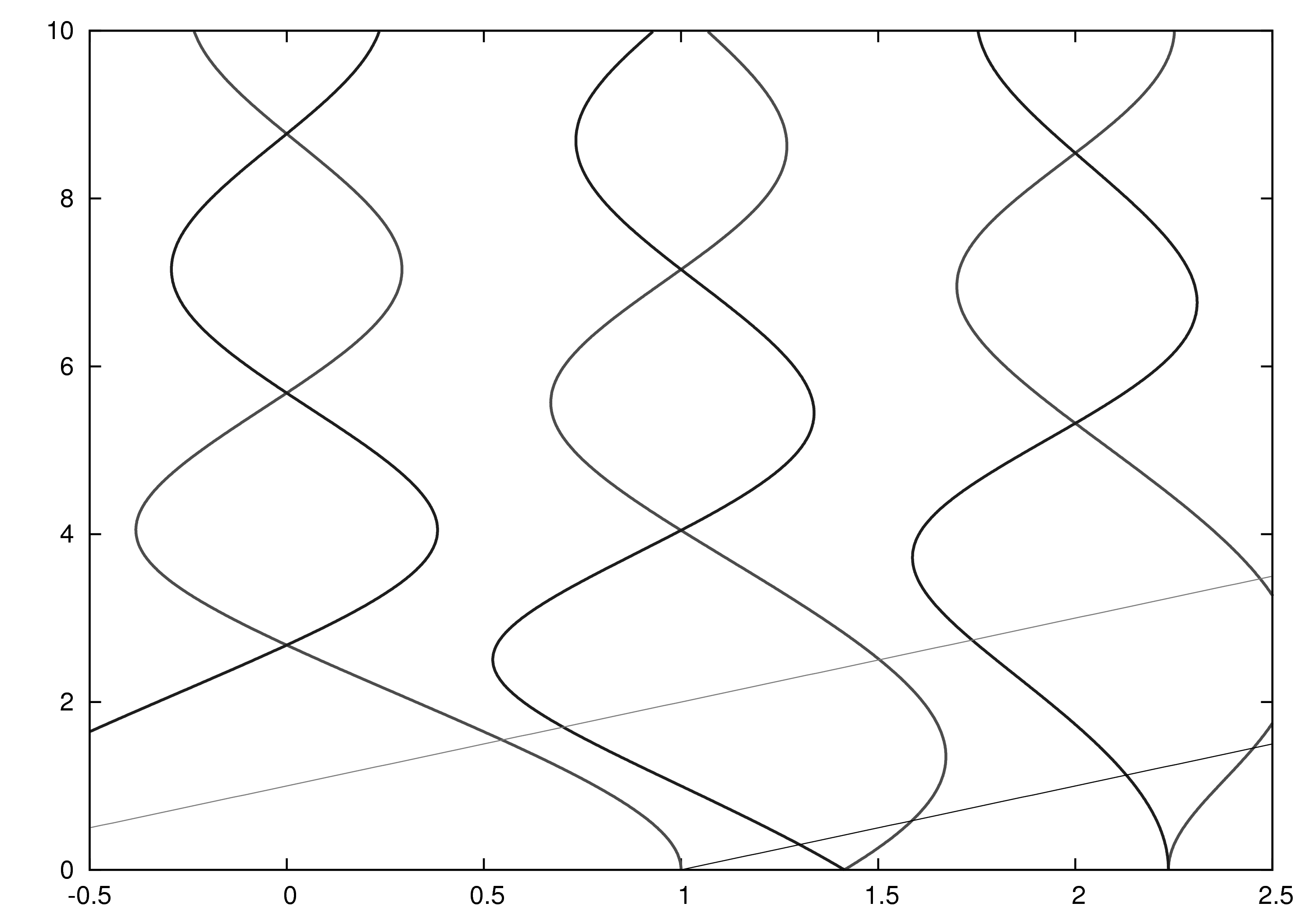}
        \caption{Здесь изображены языки Арнольда с номерами $0,1,2$ при $\mu=1$: видно, что языки обладают перемычками, лежащими на одной вертикали}\label{intersections}
	\end{center}
\end{figure}

\section{Основные результаты}\label{sec:mainresults}
\subsection{Асимптотическое приближение границ при больших амплитудах}\label{subsection:Bessel}
Напомним классическое определение целой бесселевой функции:
\begin{definition}
Функция Бесселя, соответствуюшая числу $k~\in~\mathbb{Z}$, это
$$J_k(b)=\frac{1}{2\pi} \int_0^{2\pi} \cos (kt-b \sin t) dt$$
\end{definition}

\begin{theorem}[\cite{ORK}]\label{My_Theorem}
Графики функций $a_{0,k}(b)$ и $a_{\pi,k}(b)$, определяющих границы $k$-го языка Арнольда для уравнения \eqref{eq:Joseph} в первом порядке аппроксимации ведут себя как бесселевские функции. Выполнено следующее:
\begin{equation*}
\begin{array}{l}
a_{0,k}- k \mu=-J_k \left(-\frac{b}{\mu}\right) + H_0(\mu) \times o(\frac{1}{\sqrt{b}}), b \rightarrow \infty\\
a_{\pi,k}- k \mu=J_k \left(-\frac{b}{\mu}\right) + H_{\pi}(\mu) \times o(\frac{1}{\sqrt{b}}), b \rightarrow \infty.
\end{array}
\end{equation*}
где $$H_{\varepsilon}(\mu) \leq \max \left(\nu^1_{\varepsilon}, \nu^2_{\varepsilon} \mu, \frac{\nu^3_{\varepsilon}}{\mu}, \frac{\nu^4_{\varepsilon}}{\mu^{2/3}}\right)$$
Здесь $\varepsilon=0,\pi$ and $\nu^i_\varepsilon, i=1,2,3,4$ -некоторые положительные константы и $o(\frac{1}{\sqrt{b}})$ не зависит от $\mu$.
\end{theorem}

Данные асимптотические приближения без оценки на остаточный член были ранее получены в~\cite{BKT-junc}, там же была сформулирована задача об оценке остаточного члена. Доказательство этого результата технически громоздко, и мы его не приводим здесь. Подробности см. в~\cite{ORK}.

\subsection{Быстро-медленные эффекты при малой частоте внешнего сигнала}
С помощью алгоритма, описанного в разделе \ref{method}, были получены диаграммы зон резонансного захвата для различных значений параметра $\mu$ (см. рис.\ref{pic020} и рис. \ref{pic010}). На них видно, что с уменьшением $\mu$ языки уменьшаются по ширине и приближаются друг к другу; при этом в фиксированной (не зависящей от $\mu$) окрестности нуля становятся явно выраженными три области в пространстве параметров с различным поведением языков:
\begin{itemize}
	\item  \textbf{Область $A$}: $b<a-1$. Языки тонкие.

	\item \textbf{Область $B$}: $a-1<b<a+1$. Языки заполняют практически всё пространство параметров, перемычки отсутствуют.

	\item \textbf{Область $C$}: $b>a+1$. Языки образуют сетчатую (паркетную) структуру, заполняя почти всё пространство параметров, наблюдаются перемычки.
\end{itemize}
При больших $b$ границы языков перестают приближаться друг к другу, уже не образуют явно выраженной сетчатой структуры и начинают приближаться функциями Бесселя (см. раздел~\ref{subsection:Bessel}). Область $C$, таким образом, постепенно «растворяется»: интересный открытый вопрос~--- при каких значениях $b$ (в зависимости от $\mu$) это происходит: теорема \ref{My_Theorem} дает оценку снизу на область бесселевости: $b \ge \frac{c}{\mu}$.

\begin{figure}
	\begin{center}
		\includegraphics[scale=0.11]{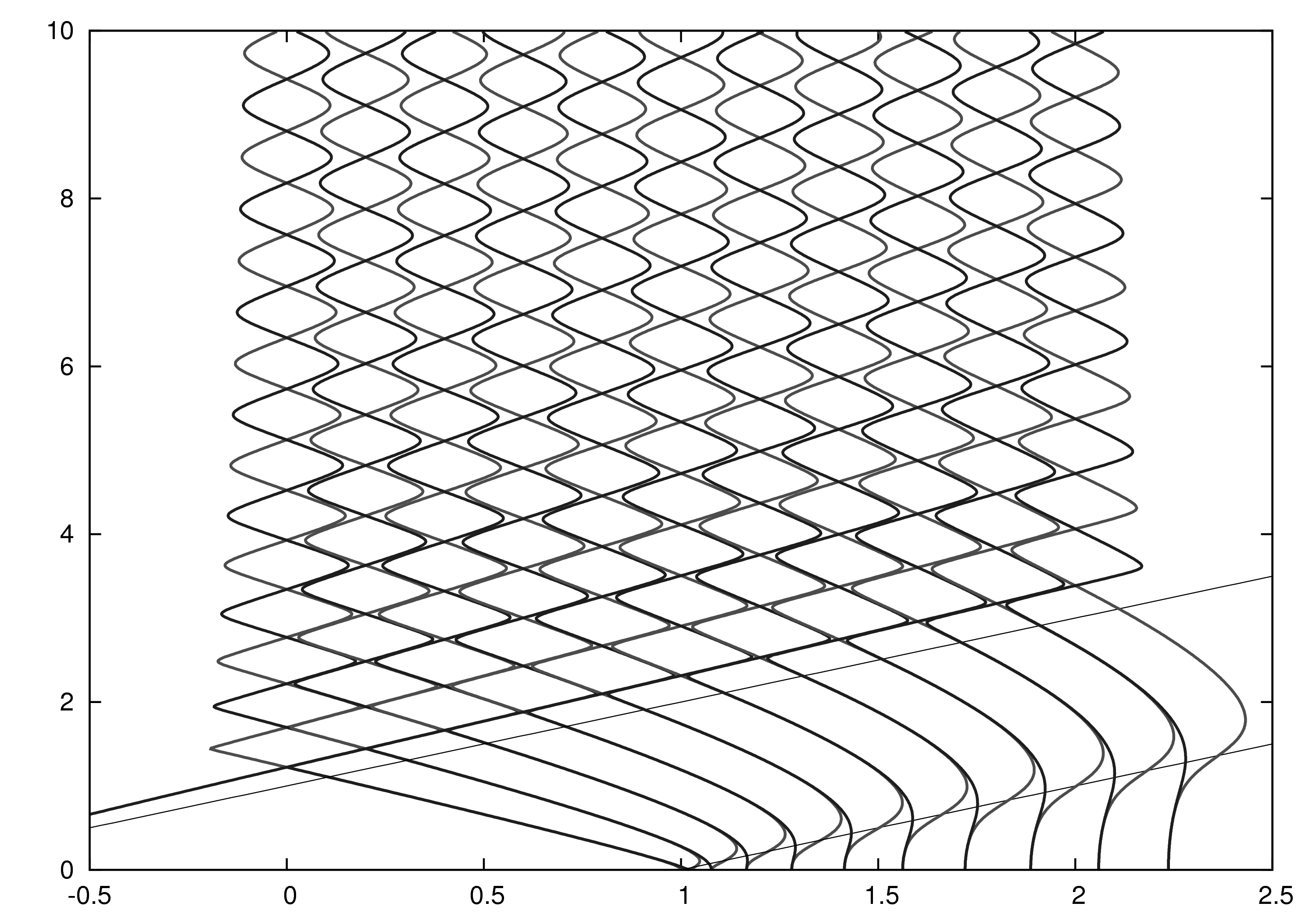}
        \caption{Линиями изображены границы зон резонансного захвата с номерами $k=1, \ldots 10$, $\mu=0.2$}\label{pic020}
        \includegraphics[scale=0.11]{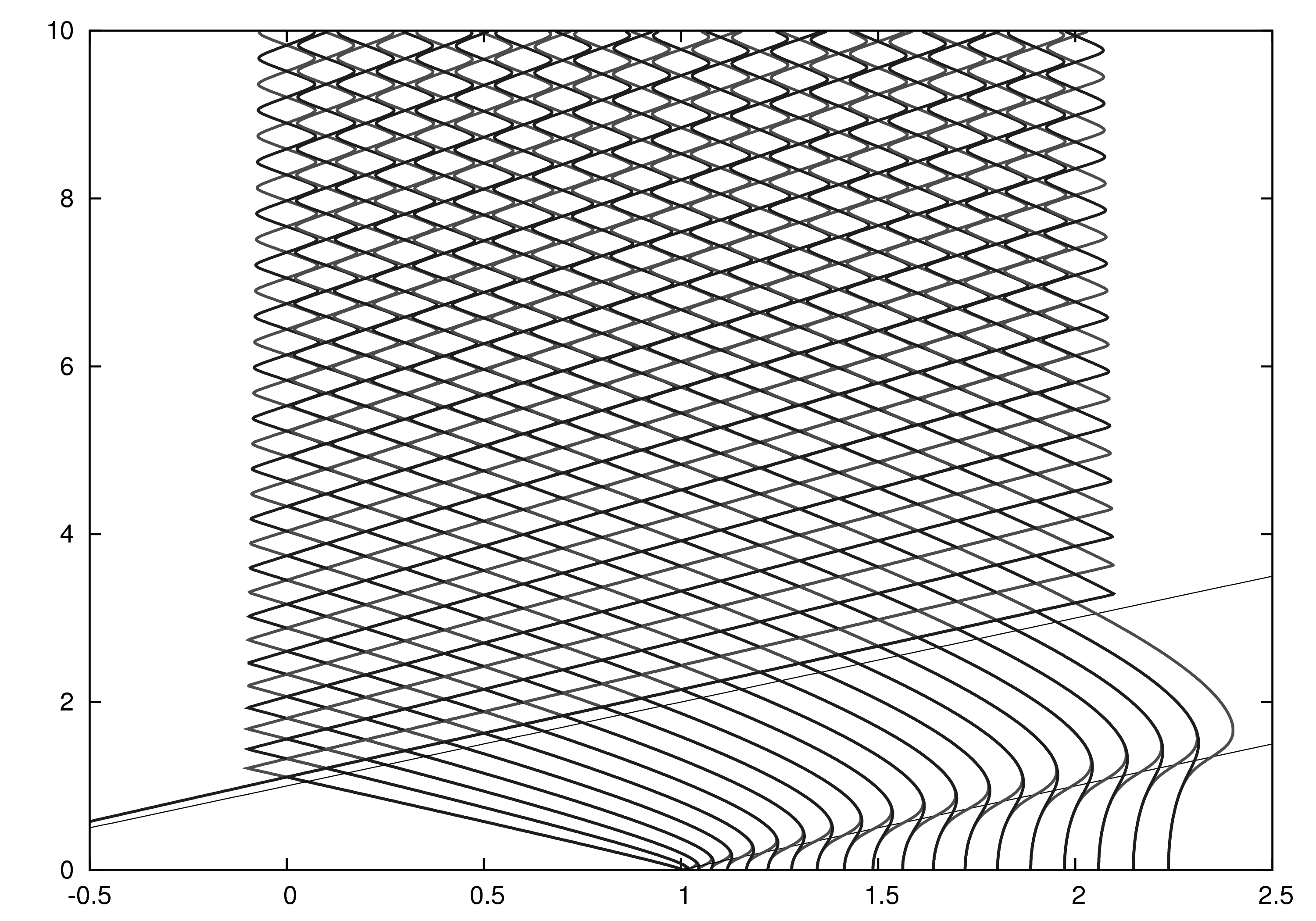}
		\caption{Здесь $\mu=0.1$ и по сравнению с $\mu=0.2$ языков на том же промежутке по $a$ становится больше: $k=1, \ldots 20$, промежутки между языками уже неразличимы на глаз}\label{pic010}
	\end{center}
\end{figure}

Структуру языков в областях $A$, $B$, $C$ можно объяснить с точки зрения теории \emph{быстро-медленных систем}. Напомним основные понятия.
\begin{definition}
	Рассмотрим семейство дифференциальных уравнений
	\begin{equation}\label{eq:slowfast}
		\begin{cases}
			\dot x=f(x,y,\eps)\\
			\dot y=\eps g(x,y,\eps),
		\end{cases}
	\end{equation}
	где $\eps\in \mathbb R_{>0}$, переменные $x$ и $y$ могут быть многомерными. Такое семейство называется \emph{быстро-медленной системой}. Переменная $x$ называется \emph{быстрой}, а переменная $y$~--- \emph{медленной}.
\end{definition}
В типичной точке фазового пространства при малых $\eps$ скорость изменения переменной $x$ много больше, чем скорость изменения переменной $y$. Это объясняет терминологию. При $\eps=0$, система~\eqref{eq:slowfast} превращается в семейство уравнений на $x$: переменная $y$ становится параметром. Такая система называется~\emph{быстрой}.
\begin{definition}
	Множество неподвижных точек
	$$M=\{(x,y)\mid f(x,y,0)=0\}$$
	быстрой системы называется~\emph{медленной поверхностью} или, в двумерном случае, \emph{медленной кривой}.
\end{definition}
Типичная траектория типичной быстро-медленной системы с одномерной быстрой и одномерной медленной переменной допускает следующее описание~\cite{MR}: это чередующиеся фазы медленного (со скоростью порядка $O(\eps)$) дрейфа вблизи медленной кривой и быстрых «срывов» (скорость порядка $O(1)$) вдоль траекторий $y=const$ быстрой системы. Срывы происходят вблизи точек, в которых касательная к медленной кривой параллельна оси быстрого движения (точек складок).

Нетрудно видеть, что семейство уравнений~\eqref{eq:Joseph_torus} можно
рассматривать как быстро-медленную систему, положив $\eps=\mu$, $g\equiv 1$.
Фактически, исследование уравнения~\eqref{eq:Joseph_torus} с точки зрения
теории быстро-медленных систем было начато в работе~\cite{GI}.

Несложным вычислением доказывается следующее утверждение о форме медленной кривой в зависимости от значений параметров (см. рис. \ref{slowcurvepic} и рис. \ref{tongueswithslowcurve}):

\begin{proposition}
В области $A$ медленная кривая системы~\eqref{eq:Joseph_torus} отсутствует; в области $B$ она имеет вид стягиваемой выпуклой кривой; имеющей ровно две точки складки, в области $C$ она распадается на пару нестягиваемых кривых с гомотопическим типом $(1,0)$, каждая из которых имеет две точки складки.
\end{proposition}

\begin{figure}
	\begin{center}
		\includegraphics[scale=0.45]{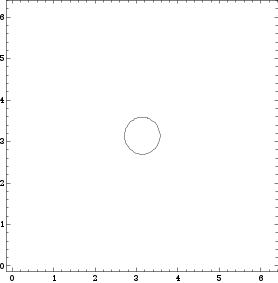}
        \includegraphics[scale=0.45]{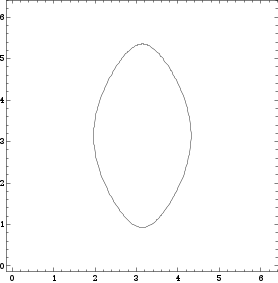}	
        \includegraphics[scale=0.45]{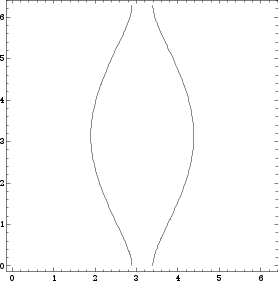}	
         \includegraphics[scale=0.45]{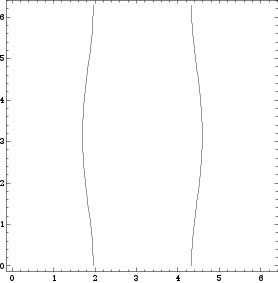}
\caption{Мультфильм об изменении медленной кривой при фиксированном $a$ и увеличении $b$: медленная кривая появляется при $b=a-1$, растет и при $b=a+1$ начинает пересекать саму себя: далее она распадается на две окружности, которые при увеличении $b$ стремятся к меридианам тора}\label{slowcurvepic}
	\end{center}
\end{figure}

\begin{theorem}\label{thm:B}
	Пусть $B'$~--- некоторое открытое ограниченное множество в пространстве параметров $(a,b)$, замыкание которого компактно вложено в область $B$. Для достаточно малого $\mu$, расстояние между соседними языками в области $B'$ не превосходит $C_1 \exp(-C_2/\mu)$ для некоторых положительных констант $C_1$, $C_2$.
\end{theorem}
\begin{proof}
Из пункта~\ref{sssec:sym} следует, что границы языков можно находить из условия, что траектория с начальным условием $(x=x_0,t=0)$ проходит через точку $(x=x_0, t=2\pi)$, где $x_0=0$ или $x_0=\pi$. На универсальной накрывающей эти условия будут иметь вид
\begin{align}\label{ga:cond:0}
	\tilde{x}_0(2\pi)&=2\pi k,\\ \label{ga:cond:pi}
	\tilde{x}_\pi(2\pi)&=\pi+2\pi k,
\end{align}
где $k\in\mathbb Z$~--- номер языка Арнольда, $\tilde{x}_0(t)$ (соответственно, $\tilde{x}_\pi(t)$)~--- фазовая кривая уравнения~\ref{eq:Joseph_torus} с начальным условием $\tilde{x}_0(0)=0$ (соответственно, $\tilde{x}_\pi(0)=\pi$), поднятая на универсальную накрывающую.

Из соображений симметрии следует, что если отображение Пуанкаре за полный период ($2\pi$) сдвигает некоторую точку на $2\pi k$ для целого $k$, то отображение Пуанкаре за половину периода будет сдвигать эту же точку на вдвое меньшую величину $\pi k$. Следовательно, условия~\eqref{ga:cond:0}-\eqref{ga:cond:pi} могут быть записаны в виде
\begin{align} \label{ga:cond:half:0}
	\tilde{x}_0(0)&=\pi k,\\ \label{ga:cond:half:pi}
	\tilde{x}_\pi(\pi)&=\pi+\pi k,
\end{align}
Это означает, что одна из границ языка с номером $k$ задаётся условием «0 переходит в 0 или $\pi$ по модулю $2\pi$», а другая~--- условием «$\pi$ переходит в $\pi$ или $0$ по модулю $2\pi$» (в зависимости от чётности $k$) за половину периода.

Заметим, что когда при непрерывном изменении параметров $a$ и $b$ значение $x_0(\pi)$ (соотв., $x_\pi(\pi)$) непрерывно меняется от $0$ до $\pi$ (соотв., от $\pi$ до $2\pi=0\pmod 2\pi$), сдвиг за половину периода увеличивается на половину оборота, а значит за полный период~--- на полный оборот, то есть число вращения увеличивается на 1. Это соответствует переходу к соседнему языку.

Пусть параметры $(a_0,b_0)\in B'$ лежат на границе языка. Без ограничения общности, можно считать, что при этом выполняется условие «0 переходит в $\pi$» (другие условия рассматриваются аналогично).

Рассмотрим дугу $J^u=[(\pi,\pi),(2\pi,\pi)]\subset\{t=\pi\}$, содержащую точку $x=3\pi/2$. Она пересекает отталкивающую часть медленной кривой. Обратим время: отталкивающая часть станет притягивающий. Образ $D$ дуги $J^u$ под действием отображения Пуанкаре с трансверсали $t=\pi$ на трансверсаль $t=0$ в обратном времени имеет длину $O(\exp(-C/\mu))$ для некоторого $C>0$. Это следует из того факта, что при движении вблизи \emph{устойчивой} части медленной кривой траектории быстро-медленной системы экспоненциально притягиваются друг к другу, подробное доказательство см. в~\cite[Proposition 4]{GI}, точную оценку для $C$ см. в~\cite[Лемма 5.4]{Ilemma}. Мы предположили, что выполняется условие «0 переходит в $\pi$» и значит нижний конец $J^u$ попадает в 0. Следовательно, $D=[0,\xi]$, где $\xi=O(\exp(-C/\mu))$.

Нетрудно показать, что производная решения по параметрам $a$ и $b$ в области $t\in [0,\pi]$ отделена от нуля (и на самом деле имеет порядок $O(1/\mu)$). Следовательно, изменяя параметр $a$ или $b$ на величину порядка $O(\exp(-C/\mu))$, можно перевести верхний конец отрезка $D$ в $0$. При этом $x_0(\pi)$ непрерывно сдвигается от $\pi$ до $2\pi$, что соответствует увеличению числа вращения на 1, то есть переходу на границу соседнего языка.
\end{proof}
\begin{theorem}\label{thm:C}
	Пусть $C'$~--- некоторое открытое ограниченное множество в пространстве параметров $(a,b)$, замыкание которого компактно вложено в область $C$. Для достаточно малого $\mu$ расстояние между соседними языками в области $C'$ не превосходит $C_1 \exp(-C_2/\mu)$ для некоторых положительных констант $C_1$, $C_2$.
\end{theorem}
\begin{proof}
	Основные соображения аналогичны тем, что использовались в доказательстве теоремы~\ref{thm:B}.
	Пусть параметры $(a_0,b_0)\in B'$ лежат на границе языка. Без ограничения
	общности, можно считать, что при этом выполняется условие «0 переходит в
	$\pi$» (другие условия рассматриваются аналогично).

	Рассмотрим трансверсаль $\Gamma=\{t=\alpha\}$, где $\alpha\in [0,\pi]$ выбрано таким образом, чтобы $\Gamma$ пересекала медленную кривую в двух точках (была отделена от точек складок). Разобьем $\Gamma$ на два полуинтервала, один из которых ($J^s$) пересекает устойчивую часть медленной кривой, а другой ($J^u$)~--- неустойчивую:
	\begin{gather*}
		J^s=\{(x,t)\mid t=\alpha,\ x\in[0,\pi)\}\subset \Gamma,\\
		J^u=\{(x,t)\mid t=\alpha,\ x\in[\pi,2\pi)\}\subset \Gamma.
	\end{gather*}
	Обозначим также через $D^u$ образ полуинтервала $J^u$ под действием отображения Пуанкаре с трансверсали $\Gamma$ на трансверсаль $t=0$ в обратном времени, а через $D^s$~--- образ полуинтервала $J^s$ под действием отображения Пуанкаре с трансверсали $\Gamma$ на трансверсаль $t=\pi$ в прямом времени. Отрезки $D^u$ и $D^s$ экспоненциально узкие по соображениям, обсуждавшимся выше (см. доказательство теоремы~\ref{thm:B}).

	Рассмотрим траекторию $x(t)$ (соотв., $x_\pi(t)$), проходящую через точку $(0,0)$ (соотв., $(\pi,0)$). Возможны два случая:

	1. Точка $0$ лежит в полуинтервале $D^u$, а значит $x_0(\alpha)\in J^u$.

	2. Точка $0$ не лежит в полуинтервале $D^u$, а значит $x_0(\alpha)\in J^s$ и $x_0(\pi)=\pi \in D^s$ (мы предположили, что это выполняется условие «0 переходит в $\pi$»).

	Предположим, что имеет место случай 2. Пусть $\pi \not\in D^u$ и значит $x_\pi(\alpha)\in J^s$ и $x_\pi(\pi)\in D^s$. В этом случае расстояние между $\pi=x_0(\pi)$ и $x_\pi(\pi)$ экспоненциально мало, и экспоненциально малым изменением параметров можно добиться выполнения условия «$\pi$ переходит в $\pi$». При этом число вращения увеличится или уменьшится на 1.

	Если $\pi\in D^u$, можно экспоненциально мало пошевелить один из параметров $a$ или $b$, чтобы это условие нарушилось.

	Случай 1 рассматривается аналогично, с заменой индексов $u$ на $s$ и наоборот.
\end{proof}
\begin{remark}
	Из доказательства теоремы~\ref{thm:C} следует, что любая точка, лежащая на границе языка, удовлетворяет одному из условий: $\{0,\pi\} \cap D^u\ne \varnothing$ или $\{0,\pi\}\cap  D^s\ne \varnothing$. Эти условия задают два семейства экспоненциально узких трубок в пространстве параметров, внутри которых лежат границы языков. Эти трубки образуют сетчатую структуру, которую можно видеть в области $C$ в численных экспериментах.

При движении вдоль границы языка, соответствующая характеристическая траектория проходит вблизи устойчивой или неустойчивой части медленной кривой, в зависимости от того, какой из случаев 1 или 2 реализуется. Это соответствует движению «влево» или «вправо». В тот момент, когда граница совершает «поворот», траектория проводит сравнимое время вблизи устойчивой и неустойчивой части медленной кривой. Такие решения называются «уточными» (см. \cite{AAIS}).
\end{remark}

\begin{figure}
	\begin{center}
		\includegraphics[scale=0.15]{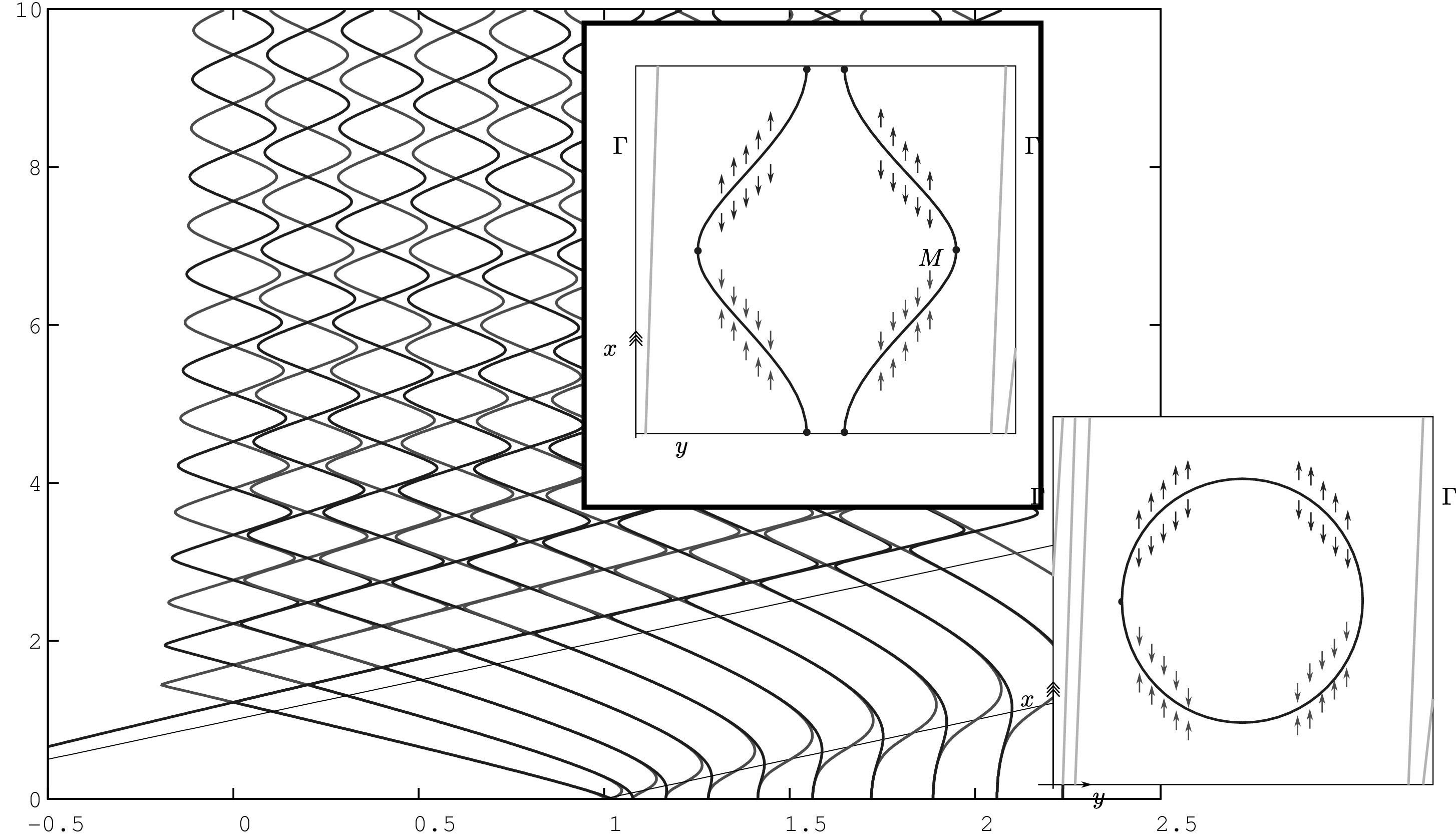}
\caption{Области $B$ и $C$ соответствуют непустым медленным кривым: в них применяется техника теории быстро-медленных систем}\label{tongueswithslowcurve}
	\end{center}
\end{figure}

\subsection{Моделирование границ: метод Ньютона}\label{method}
Учитывая наличие физических приложений, актуальной является задача построения языков Арнольда для уравнения~\eqref{eq:Joseph_torus} с помощью численных методов. В общем случае, такая задача вычислительно сложна: для нахождения числа вращения по формуле~\eqref{eq:rho} требуется численно интегрировать уравнение~\eqref{eq:Joseph_torus} на длительных промежутках времени. Так, чтобы найти число вращения с точностью $\pm \eps$, требуется проинтегрировать уравнение на промежутке времени порядка $(\eps\mu)^{-1}$. Для построения языков Арнольда требуется находить значение числа вращения на достаточно густой сетке в пространстве параметров.

Однако, используя свойства уравнения~\eqref{eq:Joseph_torus}, описанные в разделе~\ref{ss:prop}, можно предложить гораздо более эффективный алгоритм построения языков. Его описанию будет посвящена оставшаяся часть настоящего раздела.

Напомним, что границы языков Арнольда определяются условиями на образы точек 0 и $\pi$ под действием отображения Пуанкаре за половину периода (см.~\eqref{ga:cond:half:0}--\eqref{ga:cond:half:pi}). Рассмотрим условие «0 переходит в 0» (и значит $k=2l$ чётно); остальные условия рассматриваются аналогично. Пусть $x=x_0(t;a,b,\mu)$ задаёт фазовую кривую, проходящую через точку $(0,0)$.

Зафиксируем некоторое $\mu$ и положим
\begin{equation}
	Q(a,b)=x_0(\pi;a,b,\mu)
\end{equation}
Нас интересует $\pi k$-линия уровня функции $Q(a,b)$:
$$L_k:=\{(a,b)\mid Q(a,b)=\pi k=2\pi l\}.$$
Пусть граница $L_k$ задаётся графиком функции $a=a(b)$. При $b=0$, уравнение~\eqref{eq:Joseph_torus} интегрируется в квадратурах, и значение $a(0)$ можно получить явно:
\begin{equation}
	a(0)=\sqrt{1+l^2\mu^2}
\end{equation}
	Пусть найдено значение $a_0=a(b_0)$ для некоторого $b_0$. Возьмем некоторый маленький шаг $h$ и найдём приблизительно значение $a(b_0+h)$. Это означает, что нам нужно решить уравнение
	\begin{equation}\label{eq:eq}
	Q(a,b_0+h)-\pi k=0
\end{equation}
	относительно $a$. В качестве нулевого приближения положим $a=a_0$.
	Рассмотрим систему
	\begin{equation}\label{eq:var}
	\left\{
	\begin{array}{l}
	x'=\cos x +a + b \cos t,\\
	t'=\mu,\\
	u'=-u\sin x+1.
	\end{array}
	\right.
	\end{equation}
	Третье уравнение является уравнением в вариациях по параметру, $u=\frac{\partial x}{\partial a}$.
		
	Путём численного интегрирования системы~\eqref{eq:var}, найдём $Q_0=Q(a_0,b_0+h)$ и $Q'=\frac{\partial Q}{\partial a}|_{a_0,b_0+h}$. Заменяя $Q$ как функцию от $a$ на касательную в точке $a_0$ (то есть применяя один шаг метода Ньютона для нахождения корня уравнения~\eqref{eq:eq}), находим в качестве первого приближения для $a$:
	\begin{equation}
		a_1=a_0-(Q_0-\pi k)/Q'.
	\end{equation}
	После нахождения $a_1$, делаем замену $b_0+h\mapsto b_0$, $a_1\mapsto a_0$ и повторяем процедуру. Таким образом, можем найти $a(b)$ для любых значений $b$ на сетке с шагом $h$.
	
	Несмотря на то, что на каждом шаге по $b$ мы делаем только один шаг метода Ньютона, погрешность не накапливается: по индукции легко доказать, что на каждом шаге погрешность нулевого приближения составляет $O(h)$, а погрешность первого приближения составляет $O(h^2)$, что в свою очередь гарантирует, что погрешность нулевого приближения на следующим шаге составляет $O(h^2+h)=O(h)$ и т.д.

	Описанный алгоритм эффективно работает при $\mu$ порядка $1$, однако при малых значениях $\mu$ (порядка $0.1$) возникают проблемы со сходимостью метода Ньютона. Дело в том, что рассматриваемая траектория в этом случае проходит вблизи \emph{отталкивающей} части медленной кривой $M$. При этом накапливается большая производная по начальному условию и параметрам, что приводит к вычислительной неустойчивости метода.

	При нахождении в области $B$ данная проблема решается путём рассмотрения отображения $Q^{-1}$ вместо $Q$ (то есть интегрирования системы~\eqref{eq:var} в обратном времени). В этом случае рассматриваемая траектория будет проходить вблизи притягивающей части медленной кривой и описанные проблемы не возникают.

	В то же время, в области $C$ данный приём не срабатывает: при прохождении границы вблизи точки перемычки соседнего языка рассматриваемая траектория является~\emph{уточной}, то есть в прямом и обратном времени проходит вблизи отталкивающей части медленной кривой. Для решения этой проблемы используется адаптивный алгоритм, совершающий несколько шагов метода Ньютона, и в случае отстуствия сходимости использующий более устойчивый (хотя и менее эффективный) метод бисекции отрезка.

	С помощью описанного адаптивного алгоритма удаётся получать границы языков Арнольда для малых значений параметра $\mu$ вплоть до $\mu=0.01$. Следует отметить, что ранее в литературе были описаны алгоритмы, позволяющие получать языки Арнольда лишь для $\mu$ порядка $0.1$ (см.~\cite{L,BKT-junc}).
	
	\section{Благодарности}
	Результаты настоящей работы были получены в ходе многочисленных обсуждений в семинаре по динамическим системам под руководством Ю.~С.~Ильяшенко. На уравнение~\eqref{eq:Joseph} обратил наше внимание В.~М.~Бухштабер. Он также рассказал о задачах, связанных с этим уравнениям и интересных с физической точки зрения. А.~В.~Клименко заметил, что соображения симметрии мгновенно дают простые условия на границы языков в терминах решений, проходящих через $0$ или $\pi$. (Это было также обнаружено В.~М.~Бухштабером.) Эволюцию этих решений при движении вдоль границы языка и появление уточных траекторий в ходе численных экспериментов проследил Д.~А.~Филимонов. А.~А.~Глуцюк доказал целочисленность абсцисс точек перемычек. Мы выражаем сердечную признательность всем коллегам, без которых эта работа была бы невозможной.

	Все авторы поддержаны грантами РФФИ и РФФИ-CNRS: RFBR $10$-$01$-$00739$-a ; RFBR/CNRS $10$-$01$-$93115$-НЦНИЛ а.

В данной статье использованы результаты, полученные в ходе выполнения проекта №11-01-0239 <<Инвариантные многообразия и асимптотическое поведение быстро-медленных отображений>> в рамках программы «Научный фонд НИУ ВШЭ» в 2012/2013 гг.

И.Щуров также поддержан грантом Президента РФ МК-2790.2011.1: «Топологические и символические модели для одномерных и двумерных динамических систем» (2011~--- 2012)

\end{document}